\documentclass[11pt, a4paper]{article}

\usepackage[margin=3cm]{geometry}
\usepackage{amsfonts}
\usepackage{amsthm}
\usepackage{amsmath}
\usepackage{amssymb}
\usepackage[T1]{fontenc}
\usepackage{cite}
\usepackage{mathrsfs}
\usepackage{amscd}
\usepackage[utf8]{inputenc}
\usepackage{t1enc}
\usepackage{lmodern}
\usepackage{dsfont}
\usepackage{xcolor}
\usepackage{float}
\usepackage{cancel}
\usepackage{hyperref}
\usepackage[mathscr]{eucal}
\usepackage{indentfirst}
\renewcommand{\epsilon}{\varepsilon}
\usepackage{bbm}
\usepackage{graphicx}

\usepackage{tcolorbox}

\theoremstyle{plain}
\newtheorem{theorem}{Theorem}[section]

\theoremstyle{definition}

\newtheorem{remark}[theorem]{Remark}
\newtheorem{example}[theorem]{Example}

\numberwithin{equation}{section}

\title{Uniqueness of the infinite cluster for monotone percolation models without insertion tolerance}
\date{\today} 
\author{Christoforos Panagiotis\footnote{University of Bath, \nolinkurl{cp2324@bath.ac.uk}} \and Alexandre Stauffer\footnote{King's College London, \nolinkurl{a.stauffer@kcl.ac.uk}}}
\begin{document}

\maketitle

\begin{abstract}
We consider a broad class of dependent site-percolation models on $\mathbb{Z}^d$ obtained by applying a monotone automaton to a random initial particle configuration drawn from a stochastically increasing family of measures. We prove that whenever the underlying particle configuration is sampled from an insertion-tolerant measure and the avalanches generated by the dynamics produce connected sets, the supercritical phase almost surely contains a unique infinite cluster. Our result applies to several well-studied interacting particle systems, including the Abelian sandpile, activated random walk, and bootstrap percolation. In these models, the induced percolation measure typically does not satisfy standard conditions such as finite energy or insertion tolerance, so the classical Burton-Keane argument does not apply. As an application, we answer a question of Fey, Meester, and Redig \cite{FMR09} concerning percolation of the toppled vertices in the Abelian sandpile with i.i.d.\ initial configuration.
\end{abstract}

\section{Introduction}

In the current article, we consider a general class of site-percolation models $\omega^p\in \{0,1\}^{\mathbb{Z}^d}$ that can be obtained via a monotone automaton acting on an initial random configuration of particles sampled from a one-parameter family of stochastically increasing measures $\mathbb{P}_p$; see Section~\ref{sec:Def} for the precise definition of the model. This framework encompasses several fundamental interacting particle systems, including the Abelian sandpile, the activated random walk, and bootstrap percolation. These models are examples of systems exhibiting self-organized criticality and absorbing-state phase transitions, and have been the subject of extensive study; see \cite{FMR09,Jarai, LP09,RS12, Bollobas,Rolla20}. The interested reader can consult \cite{Grimmett99, LyonsPeres} for a comprehensive introduction to percolation theory.

As it is standard in percolation theory, we view $\omega^p$ as a random subgraph of $\mathbb{Z}^d$, and we are interested in its connectivity properties, in particular, whether it contains an infinite cluster. We write $\mathcal{N}$ for the number of infinite clusters of $\omega^p$ and define 
\[
p_c:=\inf\{p\in [0,1]: \, \mathbb{P}_p(\mathcal{N}\geq 1)>0\}.
\]

One of the most fundamental questions about the behaviour of the model in the supercritical regime, where an infinite cluster exists, concerns the number of infinite clusters. For Bernoulli percolation, in which each vertex is present independently, it was established by Aizenman, Kesten, and Newman \cite{AKN87} that there is a unique infinite cluster. Shortly thereafter, Burton and Keane \cite{BK89} introduced a more robust argument that also applies to dependent percolation models satisfying the \emph{finite-energy property}, which, roughly speaking, asserts that vertices can be added or removed at finite cost. More generally, their approach extends to \emph{insertion-tolerant} percolation measures, for which vertices can only be added at finite cost---see for instance \cite{ADCS15, GPPS25} for a proof of uniqueness under the assumption of insertion tolerance. Related works include \cite{GGR88, GKR88}, and \cite{HJ06} provides a survey on the uniqueness of the infinite cluster.

For percolation models without insertion tolerance, uniqueness of the infinite cluster is much more challenging to prove. Models such as the vacant set of random interlacements require some effort to handle \cite{Tex09}, while for other percolation models---such as Bernoulli line percolation \cite{HS19}---uniqueness of the infinite cluster remains, to the best of our knowledge, open.

Our motivation for studying the question of uniqueness of the infinite cluster in the present context arises from the work of Fey, Meester, and Redig \cite{FMR09}, who studied the percolative properties of the Abelian sandpile with a random initial configuration and posed several questions concerning its behaviour, including whether the infinite cluster of toppled vertices is unique when the initial configuration is sampled from a product measure. The models we consider in this article do not, in general, satisfy the insertion tolerance property, since adding a new particle to the system can potentially create an infinite avalanche. As a consequence, the standard uniqueness arguments do not apply. Nevertheless, we show that if the underlying measure $\mathbb{P}_p$ is i.i.d.\ or, more generally, is insertion-tolerant, and if each avalanche gives rise to a connected set, then the infinite cluster is unique. The latter property is satisfied by broad families of dynamics. This result, in particular, answers the question of Fey, Meester, and Redig for the Abelian sandpile \cite{FMR09}.

Below we state our main result. See Section~\ref{sec:Def} for the definitions of properties \textbf{R1}-\textbf{R4} and \textbf{D1}-\textbf{D4}.

\begin{theorem}\label{thm:main result}
Let $\mathbb{P}_p$ be a measure satisfying \textbf{R1}-\textbf{R4} and \textbf{D1}-\textbf{D4}. For every $p>p_c$, we have 
$$
\mathbb{P}_p(\mathcal{N}=1)=1.
$$
\end{theorem}

Our proof consists of three steps. Since the original model lacks insertion tolerance, we first introduce a comparison mechanism across different parameter values. More precisely, we consider parameters $p_c<p_1<p_2<p_3$ and define an auxiliary configuration $\omega^{p_1,p_2}$ that is insertion-tolerant and interpolates between $\omega^{p_1}$ and $\omega^{p_2}$. By the Burton–Keane argument, $\omega^{p_1,p_2}$ has a unique infinite cluster. Next, we relate this auxiliary configuration back to the original model. Using the mass-transport principle, we show that if there exists an $\omega^{p_3}$ infinite cluster that does not contain the $\omega^{p_1,p_2}$ infinite cluster, then the distance between them is attained infinitely many times. Finally, we exploit this property via a multi-valued map argument to merge these two infinite clusters into one, thereby contradicting the existence of multiple infinite clusters.

\paragraph{Acknowledgements:} We would like to thank Ahmed Bou-Rabee for useful discussions. CP was supported by an EPSRC New Investigator Award (UKRI1019).

\section{Definitions}\label{sec:Def}

Consider the hypercubic lattice $\mathbb{Z}^d$. Throughout the paper, we write $o$ to denote an arbitrary origin and $B_n(o)$ to denote the ball of radius $n$ centred at $o$. 

We begin by briefly introducing the Abelian sandpile model with a random initial configuration of particles, which serves as our motivating example for the properties that our dynamics are required to satisfy. Consider an initial configuration of particles $\xi \in \mathbb{N}^{\mathbb{Z}^d}$. A vertex is called \textit{stable} if $\xi_x < t$, where $t=2d$ is the degree of every vertex in $\mathbb{Z}^d$, and \textit{unstable} otherwise. If a vertex $x$ is unstable, then it \textit{topples}, meaning that it sends one particle to each of its neighbours. This procedure continues for as long as unstable vertices exist. Depending on the initial configuration $\xi$, it is possible that the process \textit{stabilizes}, in the sense that each vertex topples a finite number of times, or that it does not stabilize. If the process stabilizes, then the Abelian sandpile satisfies the so-called \textit{Abelian property}, which, roughly speaking, states that the final configuration does not depend on the order in which topplings are performed.

In \cite{FMR09}, Fey, Meester, and Redig considered the Abelian sandpile with a random initial configuration sampled from a one-parameter family of probability measures $\mathbb{P}_p$. One can then define two critical points:
$p^{\textup{topple}}_c := \inf\{p \in [0,1] : \mathbb{P}_p(\mathcal{N} \geq 1)\}$,
where $\mathcal{N}$ denotes the number of infinite clusters of toppled vertices, and
$p^{\textup{stab}}_c := \sup\{p \in [0,1] : \mathbb{P}_p(\text{stabilization}) = 1\}$.
It is expected that $p^{\textup{topple}}_c < p^{\textup{stab}}_c$ in great generality. In \cite{FMR09}, it was asked whether almost surely $\mathcal{N} = 1$ for $p>p^{\textup{topple}}_c$. The classical Burton--Keane argument \cite{BK89} does not apply in this context, due to the fact that the final configuration is not necessarily insertion tolerant. To see this, note that one can construct deterministic stable configurations such that adding a single particle creates an infinite avalanche of toppled vertices. Despite this, our main result establishes the uniqueness of the infinite cluster of toppled vertices for every $p>p^{\textup{topple}}_c$.

We now move on to the general setting to which our results apply. We begin by introducing the underlying randomness. To this end, we first set up some notation. Given $\xi \in [0,\infty)^{\mathbb{Z}^d}$ and $x \in \mathbb{Z}^d$, we define $\tau_x \xi$ to be the configuration given by $(\tau_x \xi)_y = \xi_{y-x}$ for every $y \in \mathbb{Z}^d$. We extend this notation to events $A$ that are measurable with respect to the product $\sigma$-algebra by defining $\tau_x A = \{\tau_x \xi : \xi \in A\}$. We say that an event $A$ is \emph{translation-invariant} if $\tau_x A = A$ for every $x \in \mathbb{Z}^d$. We say that a measure $\mu$ on $[0,\infty)^{\mathbb{Z}^d}$ is \emph{ergodic} if, for every translation-invariant event $A$, we have $\mu(A) \in \{0,1\}$. All these definitions extend in a natural way to $[0,\infty)^{\mathbb{Z}^d} \times [0,\infty)^{\mathbb{Z}^d}$ and $[0,\infty)^{\mathbb{Z}^d} \times [0,\infty)^{\mathbb{Z}^d} \times [0,\infty)^{\mathbb{Z}^d}$.

\textbf{Definition of underlying randomness.} Let $t > 0$ be a constant fixed throughout the paper. This constant $t$ should be thought of as the degree of the graph in the context of the Abelian sandpile. Below we denote the $\sigma$-algebra generated by some random variables $W_1,W_2,\ldots$ as $\sigma(W_1,W_2,\ldots)$. We consider a family of measures $(\mathbb{P}_p)_{p\in [0,1]}$ on $[0,\infty)^{\mathbb{Z}^d}$ such that for every $p_1<p_2<p_3$ there exists a coupling $(\mathbb{P}_{p_1,p_2,p_3}, X, Y, Z)$ with $X\sim \mathbb{P}_{p_1}$, $Y\sim \mathbb{P}_{p_2}$ and $Z\sim \mathbb{P}_{p_3}$ that satisfies the following properties:
\begin{enumerate}
    \item[$\textbf{R1}$] \textbf{(Translation invariance)} for every $x\in\mathbb{Z}^d$, $(\tau_x X,\tau_x Y, \tau_x Z)$ has law $\mathbb{P}_{p_1,p_2,p_3}$,
    \item[$\textbf{R2}$] \textbf{(Monotonicity)} $\mathbb{P}_{p_1,p_2,p_3}$-almost surely, for every $x\in\mathbb{Z}^d$ we have $X_x\le Y_x \le Z_x$,
    \item[$\textbf{R3}$] \textbf{(Insertion tolerance)} there exists a constant $\varepsilon>0$ such that $\mathbb{P}_{p_1,p_2,p_3}$-almost surely, for every $x\in\mathbb{Z}^d$ we have $\mathbb{P}_{p_1,p_2,p_3}(Y_x\geq t \mid \sigma(X,(Y_y)_{y\in \mathbb{Z}^d\setminus\{x\}}))\geq \varepsilon$ and $\mathbb{P}_{p_1,p_2,p_3}(Z_x\geq t \mid \sigma(X,Y,(Z_y)_{y\in \mathbb{Z}^d\setminus\{x\}}))\geq \varepsilon$,
    \item[$\textbf{R4}$] \textbf{(Ergodicity)} $\mathbb{P}_{p_1,p_2,p_3}$ is ergodic.
\end{enumerate}

\begin{example}
The simplest example satisfying $\textbf{R1}-\textbf{R4}$ is obtained by letting each $\mathbb{P}_p$ be a product measure $\otimes \mu_p$, where $(\mu_p)_{p\in [0,1]}$ is a family of measures on $\mathbb{N}=\{0,1,2,\ldots\}$ that is strictly stochastically increasing in the following sense: for every $p_1 < p_2$, we have $\mu_{p_2}([n,\infty)) \ge \mu_{p_1}([n,\infty))$ for all $n \in \mathbb{N}$, with strict inequality $\mu_{p_2}([n,\infty)) > \mu_{p_1}([n,\infty))$ for $n = 1,2,\dots,t$. 

It is not hard to see that the following two one-parameter families satisfy the above (strict) stochastic domination condition: 
\begin{enumerate}
    \item Poisson$(\rho)$,
    \item $t\cdot\text{Ber}(p)$, i.e.\ $\mu_p(X=t)=p$ and $\mu_p(X=0)=1-p$.
\end{enumerate} 
\end{example}

Before moving to the definition of the dynamics, let us introduce some useful notation. Given a configuration $\eta\in [0,\infty)^{\mathbb{Z}^d}$ and a set of vertices $F\subset \mathbb{Z}^d$, we define $\Pi_F \eta$ to be the set of configurations $\xi\in [0,\infty)^{\mathbb{Z}^d}$ such that
$\xi_y=\eta_y$ for every $y \in \mathbb{Z}^d\setminus F$ and $\xi_y\geq t$ otherwise, where $t$ is the constant of \textbf{R3}.
We extend this notation to events $A$ by $\Pi_F A:=\{\Pi_F \eta \, ; \, \eta\in A\}$. If $F=\{x\}$, we simply write $\Pi_x \eta$ and $\Pi_x A$. 

\textbf{Definition of dynamics.} We now proceed to the definition of the dynamics. Let $T:[0,\infty)^{\mathbb{Z}^d}\rightarrow \{0,1\}^{\mathbb{Z}^d}$ be a (possibly random) map. Given a configuration $\xi\in [0,\infty)^{\mathbb{Z}^d}$, let us write $\mathcal{C}_x(\xi)$ for the connected component of $x$ consisting of vertices $y$ such that $T(\xi)_y=1$. Given $x\in\mathbb{Z}^d$ and $r\in \mathbb{R}$, we define $(\xi^{x,r})_{y\in \mathbb{Z}^d}$ as the configuration that is equal to $\max\{r,\xi_x\}$ for $y=x$, and otherwise is equal to $\xi_y$.

Below $t$ denotes the constant introduced in the definition of the underlying randomness. We assume that $T$ satisfies the following properties:
\begin{enumerate}
    \item[$\textbf{D1}$] \textbf{(Translation invariance)} for every $x\in\mathbb{Z}^d$ and $\xi\in[0,\infty)^{\mathbb{Z}^d}$, $T(\tau_x \xi)=\tau_x T(\xi)$,
    \item[$\textbf{D2}$] \textbf{(Monotonicity)} for every $\xi,\xi'\in [0,\infty)^{\mathbb{Z}^d}$ such that $\xi_x\leq \xi'_x$ for every $x\in \mathbb{Z}^d$, we have $T(\xi)_x\leq T(\xi')_x$ for every $x\in \mathbb{Z}^d$,
    \item[$\textbf{D3}$] \textbf{(Occupation threshold)} for every $x\in\mathbb{Z}^d$, if $\xi_x\geq t$, then $T(\xi)_x=1$,
    \item[$\textbf{D4}$] \textbf{(Connectivity)} for every $\xi\in [0,\infty)^{\mathbb{Z}^d}$, every $x\in \mathbb{Z}^d$ and every $r\geq 0$, we have 
    \[\{y\in \mathbb{Z}^d\setminus\mathcal{C}_x(\xi^{x,r}): |\mathcal{C}_y(\xi^{x,r})|=\infty\}\subset \{y\in \mathbb{Z}^d: |\mathcal{C}_y(\xi)|=\infty\}\]
\end{enumerate}

In words, \textbf{D4} says that increasing $\xi$ at $x$ may potentially change the structure of the cluster $\mathcal{C}_x$, but it will not affect any other \emph{infinite} cluster. Let us remark that \textbf{D4} is implied by the following simpler property:
\begin{enumerate}
    \item[\textbf{AD4}] for every $\xi\in [0,\infty)^{\mathbb{Z}^d}$, every $x\in \mathbb{Z}^d$ and every $r\geq 0$, there exists a connected set of vertices $S$ that contains $x$ such that $\{y\in \mathbb{Z}^d : T(\xi^{x,r})=1\}=\{y\in \mathbb{Z}^d : T(\xi)=1\}\cup S$.
\end{enumerate}

In words, \textbf{AD4} says that increasing $\xi$ at $x$ adds a connected set of open vertices that contains $x$.

\begin{example}
Let us explain why properties $\textbf{D1}-\textbf{D4}$ are satisfied by the toppled vertices of the Abelian sandpile and the activated random walk. In both of these models, each vertex contains a certain number of particles that move around the graph by jumping to a neighbouring vertex according to some rule. If a particle is located at some vertex $x$ and eventually it jumps to another vertex, then we say that $x$ topples. Due to the fact that particles jump between neighbouring vertices, when one adds more particles at some vertex, their trajectory is a connected set. Hence \textbf{AD4} is satisfied, and so does \textbf{D4}. A common property of both of these models is that once a vertex $x$ contains a sufficient number of particles, then it topples, which corresponds to \textbf{D3}. Moreover, both models satisfy the Abelian property, which roughly speaking states that if the configuration of particles stabilizes, i.e.\ each vertex topples finitely many times, then the final configuration does not depend on the order in which particles move. The Abelian property implies that \textbf{D2} is satisfied, as one can first stabilize the $\xi$ particles and only then stabilize the remaining $\xi'-\xi$ particles. Let us remark that in the case of the activated random walk, the function $T$ is itself random.
\end{example}

\begin{example}
Our framework includes maps $T:\{0,1\}^{\mathbb{Z}^d}\rightarrow \{0,1\}^{\mathbb{Z}^d}$ that satisfy \textbf{D1}-\textbf{D4} with $t=1$, since their definition can be trivially extended to the whole of $[0,\infty)^{\mathbb{Z}^d}$. Note that this framework includes monotone cellular automata such as bootstrap percolation. Indeed, in bootstrap percolation each vertex $x$ is occupied when $\xi_x=1$ and vacant otherwise. Occupied vertices remain occupied forever, which corresponds to \textbf{D3}, and each vacant vertex becomes occupied if sufficiently many of its neighbours are occupied, which implies \textbf{AD4}, and hence \textbf{D4}. 
\end{example}

When $\xi\sim \mathbb{P}_p$ we will write $\omega^p$ to denote $T(\xi)$ and $\mathcal{C}^p_x$ to denote $\mathcal{C}_x(\xi)$. As usual in percolation theory, we are interested in the existence of infinite clusters. Let us write $\mathcal{N}$ for the number of $\omega^p$ infinite clusters. We define $p_c=\inf\{p\in [0,1]: \mathbb{P}_p(\mathcal{N}\geq 1)>0\}$. By properties \textbf{R2} and \textbf{D2} we have that $\mathbb{P}_p(\mathcal{N}\geq 1)>0$ for every $p>p_c$.


\section{Proof of main result}

In this section, we will prove Theorem~\ref{thm:main result}. Our proof will utilise the following important fact that we now introduce.



We say that a function $f:\mathbb{Z}^d\times \mathbb{Z}^d\rightarrow [0,\infty]$ is \emph{diagonally invariant} if for every $z\in \mathbb{Z}^d$ we have $f(x-z,y-z)=f(x,y)$. The \emph{mass-transport principle} states that
\[
\sum_{y\in \mathbb{Z}^d} f(x,y)=\sum_{y\in \mathbb{Z}^d} f(-y+2x,x)=\sum_{y\in \mathbb{Z}^d} f(y,x).
\]
The mass-transport principle has been used in several works in percolation theory, especially on graphs beyond $\mathbb{Z}^d$, to study the structure of infinite clusters---see e.g.\ \cite{BLPS99, HP99, LS99}.

We are now ready to prove Theorem~\ref{thm:main result}.

\begin{proof}[Proof of Theorem~\ref{thm:main result}] We split the proof into 3 steps.

\textbf{Step 1:} Let $p>p_c$ and consider $p_c<p_1<p_2<p_3:=p$. We will work with the coupling $(\mathbb{P}_{p_1,p_2,p_3},X,Y,Z)$. Let us write $\omega^{p_1}$, $\omega^{p_2}$, $\omega^{p_3}$ for the site-percolation configurations generated by $X,Y$ and $Z$, respectively. We first define a percolation configuration that interpolates between $\omega^{p_1}$ and $\omega^{p_2}$ and has a unique infinite cluster. Define the configuration $\omega^{p_1,p_2}$ by $\omega^{p_1,p_2}(x)=1$ if $\omega^{p_1}(x)=1$ or if $Y_x\geq t$. Since there exists an infinite cluster $\omega^{p_1}$-almost surely, there exists an infinite cluster $\omega^{p_1,p_2}$-almost surely. Note also that $\omega^{p_1,p_2}\leq \omega^{p_2}\leq \omega^{p_3}$ by properties \textbf{D2} and \textbf{D3}.

Note that the law of $\omega^{p_1,p_2}$ is ergodic by \textbf{R4}, hence the number of infinite clusters of $\omega^{p_1,p_2}$ is almost-surely constant. Furthermore, the law of $\omega^{p_1,p_2}$ is insertion-tolerant by properties \textbf{R3} and \textbf{D3}, in the sense that conditionally on $(\omega^{p_1,p_2})_{y\in \mathbb{Z}^d\setminus\{x\}}$, the probability that $\omega^{p_1,p_2}(x)=1$ is almost surely strictly positive. It is not hard to see that $\omega^{p_1,p_2}$ has either a unique or infinitely many infinite clusters. Indeed, if it has $k\in (1,\infty)$ infinite clusters (almost surely), then there exists $n\geq 1$ such that with positive probability, all $k$ infinite clusters intersect $B_n(o)$. Let us denote this event $\mathcal{A}_n$.

By \textbf{R3}, we have 
$$
   \mathbb{P}_{p_1,p_2}(\Pi_{B_n(o)}\mathcal{A}_n)>0.
$$ 
Since $\mathcal{A}_n$ is an increasing function of the initial configuration $\xi$ by \textbf{D2}, we have $\Pi_{B_n(o)}\mathcal{A}_n=\{\forall x\in B_n(o), \xi_x\geq t\}\cap \mathcal{A}_n$. 
When the event $\{\forall x\in B_n(o), \xi_x\geq t\}\cap \mathcal{A}_n$ happens, all the infinite clusters merge into one infinite cluster. This contradicts the fact that the number of infinite clusters is constant. Thus, $\omega^{p_1,p_2}$ has either a unique or infinitely many infinite clusters.

The Burton-Keane argument \cite{BK89} also applies to insertion-tolerant measures; we now briefly describe it. If $\omega^{p_1,p_2}$ has infinitely many infinite clusters, then arguing as above, we can find $n\geq 1$ such that with positive probability, $\omega^{p_1,p_2}(x)=1$ for every $x\in B_n(o)$, $o$ is connected to infinity, and removing $B_n(o)$ from $\omega^{p_1,p_2}$, we get at least $3$ disjoint infinite clusters which are adjacent to $B_n(o)$. When this happens, we say that $o$ is a \emph{coarse-trifurcation}. We similarly extend the definition to any $y\in \mathbb{Z}^d$. Then we modify our lattice $\mathbb{Z}^d$ as follows. For each $y \in 4n\mathbb{Z}^d$, we identify $B_n(y)$ with a single vertex $u_y$ that we connect with an edge to each vertex adjacent to $B_n(y)$ in $\mathbb{Z}^d$. We denote this graph $G_n$. Let $L\geq 4n$, we let $\mathcal{T}_L$ be the number of vertices $y\in 4n\mathbb{Z}^d$ such that $B_{2n}(y)\subset B_L(o)$ and $y$ is a coarse-trifurcation. Note that when $y$ is a coarse-trifurcation, then $u_y$ is a \emph{trifurcation} in $G_n$, in the sense that $u_y$ is connected to infinity, and removing $u_y$, splits the cluster of $u_y$ into at least $3$ infinite clusters.
Then one can apply the standard Burton-Keane argument to the graph $G_n$ to deduce that almost surely $\mathcal{T}_L\leq CL^{d-1}$ for some $C>0$. Since $\mathbb{E}_{p_1,p_2}(\mathcal{T}_L)\geq cL^d\mathbb{P}_{p_1,p_2}(o \text
{ is a coarse-trifurcation})$ for some $c=c(n,d)>0$, sending $L$ to infinity we deduce that $\mathbb{P}_{p_1,p_2}(o \text
{ is a coarse-trifurcation})=0$. This contradiction implies that $\omega^{p_1,p_2}$ has a unique infinite cluster. 

\textbf{Step 2:} Let us write $\mathcal{C}_{\infty}^{p_1,p_2}$ for the unique $\omega^{p_1,p_2}$ infinite cluster. We will show that almost surely, for every $\omega^{p_3}$ infinite cluster $\mathcal{C}$ that does not contain $\mathcal{C}_{\infty}^{p_1,p_2}$,
the (random) distance between $\mathcal{C}$ and $\mathcal{C}_{\infty}^{p_1,p_2}$ is attained infinitely many times.
Indeed, let us assume that with positive probability this does not happen. Then we can find a vertex $u$ such that with positive probability, the cluster of $u$ under $\omega^{p_3}$ is infinite and does not contain $\mathcal{C}_{\infty}^{p_1,p_2}$. Furthermore, there exists a vertex $v$ such that with positive probability, the cluster $\mathcal{C}^{p_3}_u$ of $u$ is infinite and does not contain $\mathcal{C}_{\infty}^{p_1,p_2}$, $v\in \mathcal{C}^{p_3}_u$, and $v$ attains the distance between $\mathcal{C}^{p_3}_u$ and $\mathcal{C}_{\infty}^{p_1,p_2}$. In particular, with positive probability, $\mathcal{C}^{p_3}_v$ is infinite and does not contain $\mathcal{C}_{\infty}^{p_1,p_2}$, and $v$ attains the distance between $\mathcal{C}^{p_3}_v$ and $\mathcal{C}_{\infty}^{p_1,p_2}$.

Now we obtain a contradiction as follows. For a vertex $x$ and configurations $\omega^{p_1,p_2}$ and $\omega^{p_3}$, let $N(x,\omega^{p_1,p_2},\omega^{p_3})$ be the number of times the distance between $\mathcal{C}^{p_3}_x$ and $\mathcal{C}_{\infty}^{p_1,p_2}$ is attained, with the convention that $N(x,\omega^{p_1,p_2},\omega^{p_3})=\infty$ if $\mathcal{C}_{\infty}^{p_1,p_2}\subset \mathcal{C}^{p_3}_x$. Define $F(x,y,\omega^{p_1,p_2},\omega^{p_3})$ to be $0$ if $N(x,\omega^{p_1,p_2},\omega^{p_3})=\infty$ or if the distance between $\mathcal{C}^{p_3}_x$ and $\mathcal{C}_{\infty}^{p_1,p_2}$ is not attained at $y$, and to be $1/N(x,\omega^{p_1,p_2},\omega^{p_3})$ if $y\in \mathcal{C}^{p_3}_x$ and the distance between $\mathcal{C}^{p_3}_x$ and $\mathcal{C}_{\infty}^{p_1,p_2}$ is attained at $y$. Then $F$ is \emph{diagonally invariant} in the sense that for every $z\in \mathbb{Z}^d$, $F(x-z,y-z,\tau_z \omega^{p_1,p_2},\tau_z \omega^{p_3})=F(x,y,\omega^{p_1,p_2},\omega^{p_3})$. Since the law of $(\omega^{p_1,p_2},\omega^{p_3})$ is translation invariant, $f(x,y):=\mathbb{E}_{p_1,p_2,p_3}(F(x,y,\omega^{p_1,p_2},\omega^{p_3}))$ is also diagonally invariant, and we can apply the mass-transport principle. On the one hand, 
$$
\sum_{y\in \mathbb{Z}^d} F(x,y,\omega^{p_1,p_2},\omega^{p_3})\leq 1,
$$
hence 
$$
\sum_{y\in \mathbb{Z}^d} f(x,y)\leq 1.
$$
On the other hand, on the event $$\{|\mathcal{C}^{p_3}_y|=\infty, \mathcal{C}^{p_3}_y \, \cap \, \mathcal{C}_{\infty}^{p_1,p_2}=\emptyset, y \text{ attains } d(\mathcal{C}^{p_3}_y, \mathcal{C}_{\infty}^{p_1,p_2}), N(y,\omega^{p_1,p_2},\omega^{p_3})<\infty\}$$
we have
$$
\sum_{x\in \mathbb{Z}^d} F(x,y,\omega^{p_1,p_2},\omega^{p_3})=\infty,
$$
where $d(\cdot,\cdot)$ denotes the graph distance.
Since this event happens with positive probability,
$$
\sum_{x\in \mathbb{Z}^d} f(x,y)=\infty.
$$
This contradiction implies that the distance between any infinite cluster $\mathcal{C}$ and $\mathcal{C}_{\infty}^{p_1,p_2}$ is attained infinitely many times, as desired.

\textbf{Step 3:} We will show that almost surely, every $\omega^{p_3}$ infinite cluster contains $\mathcal{C}_{\infty}^{p_1,p_2}$, which implies that there exists a unique $\omega^{p_3}$ infinite cluster. We argue by contradiction. Let us assume that with positive probability $\mathcal{C}^{p_3}_o$ is infinite and does not contain $\mathcal{C}_{\infty}^{p_1,p_2}$. Combined with Step 2, this implies that there exist constants $D>0$ and $c>0$ such that with probability at least $c$, $\mathcal{C}^{p_3}_o$ is infinite, its distance from $\mathcal{C}_{\infty}^{p_1,p_2}$ is $D$, and this distance is attained infinitely many times. Let us write $\mathcal{E}$ to denote the latter event.

Let $N$ and $K$ be large enough constants to be defined.
Recall the definition of the map $T$. Given configurations $\xi_{p_1}\leq \xi_{p_2}\leq \xi_{p_3}$, by abuse of notation we define $T(\xi_{p_1},\xi_{p_2})_x:=\max\{T(\xi_{p_1})_x, \mathbbm{1}_{\xi_{p_2}(x)\geq t}\}$ , and we write $(\xi_{p_1},\xi_{p_2},\xi_{p_3})\in \mathcal{A}$ for some event $\mathcal{A}$ on $\{0,1\}^{\mathbb{Z}^d}\times \{0,1\}^{\mathbb{Z}^d}$ to denote that $(T(\xi_{p_1},\xi_{p_2}), T(\xi_{p_3}))\in \mathcal{A}$. Now given $(\xi_{p_1}, \xi_{p_2},\xi_{p_3})\in \mathcal{E}$, write $\mathcal{S}$ for the restriction of $\mathcal{C}_{\infty}^{p_1,p_2}$ to $B_N(o)$, and $\overline{\mathcal{S}}^D$ for the set of vertices $x\in B_N(o)$ such that $d(x,\mathcal{S})\leq D-1$. We define $\tilde{\xi}_{p_3}$ to be the configuration which is equal to $0$ for $x\in \overline{\mathcal{S}}^D$, and is equal to the initial configuration $\xi_{p_3}$ otherwise.
Notice that by \textbf{D4}, $\mathcal{C}^{p_3}_o(\tilde{\xi}_{p_3})$ and $\mathcal{C}^{p_3}_o(\xi_{p_3})$ coincide on the event $\mathcal{E}$. 

For each $u,v$ we fix a geodesic $\gamma_{u,v}$ from $u$ to $v$. Consider the event $\mathcal{E}_{N,K}$ that $\mathcal{C}^{p_3}_o(\tilde{\xi}_{p_3})$ is infinite, its distance from $\mathcal{C}_{\infty}^{p_1,p_2}$ is equal to $D$, and furthermore, this distance is attained at least $K$ times in $B_N(o)$. Here, the latter means that there are at least $K$ pairs $(x,y)$ with $x \in \mathcal{S}$, $y \in \mathcal{C}^{p_3}o \cap B_N(o)$, $d(x,y) = D$, and such that the entire geodesic $\gamma{x,y}$ lies within $B_N(o)$. Notice that when the event $\mathcal{E}$ happens, and the distance between $\mathcal{C}_{\infty}^{p_1,p_2}$ and $\mathcal{C}^{p_3}_o$ is attained at least $K$ times in $B_N(o)$, then $\mathcal{E}_{N,K}$ happens. Thus, for every $K\geq 1$, for every $N\geq 1$ large enough, we have 
\begin{equation}\label{eq: contradiction}
\mathbb{P}(\mathcal{E}_{N,K}\cap \mathcal{E})\geq c/2,
\end{equation} 
where for simplicity we write $\mathbb{P}=\mathbb{P}_{p_1,p_2,p_3}$.
We will derive a contradiction by showing that $\mathbb{P}(\mathcal{E}_{N,K}\cap \mathcal{E})< c/2$ for every $K$ large enough.

To this end, we will use the multi-valued map principle. First, we partition the event $\mathcal{A}:=\{\xi_{p_1}\leq \xi_{p_2}\leq \xi_{p_3}\in [0,\infty)^{\mathbb{Z}^d}: (\xi_{p_1},\xi_{p_2},\xi_{p_3})\in \mathcal{E}_{N,K}\cap\{\mathcal{C}_{\infty}^{p_1,p_2}\cap \mathcal{C}^{p_3}_o=\emptyset\}\}$ into events of the form $\mathcal{A}(\psi):=\{(\xi_{p_1}, \xi_{p_2},\xi_{p_3})\in \mathcal{A}: (\mathcal{C}_{\infty}^{p_1,p_2}\cap B_N(o),\mathcal{C}^{p_3}_o\cap B_N(o))=\psi\}$. Then we define a map $F$ that maps each event $\mathcal{A}(\psi)$ to a collection of at least $K$ events $F(\mathcal{A}(\psi),x,y)$, where $x\in \mathcal{S}$ and $y\in \mathcal{C}^{p_3}_o \cap B_N(o)$ are such that $d(x,y)=D$ and $\gamma_{x,y}$ lies in $B_N(o)$.
For $(\xi_{p_1},\xi_{p_2},\xi_{p_3})\in \mathcal{A}(\psi)$ and $x,y$ as above, we define $F(\xi_{p_1},\xi_{p_2},\xi_{p_3},x,y)$ to be the set of configurations $(\xi_{p_1},\xi_{p_2},\xi'_{p_3})$, where $\xi'_{p_3}$ is obtained from $(\xi_{p_1},\xi_{p_2},\xi_{p_3})$
by letting
$\xi'_{p_3}(v)\geq t$ for every $v\in \gamma_{x,y}\setminus \{x,y\}$ such that $\xi_{p_3}(v)<t$, while otherwise letting $\xi'_{p_3}(v)$ be equal to $\xi_{p_3}(v)$. Then $F(\mathcal{A}(\psi),x,y)$ is defined as the union of $F(\xi_{p_1},\xi_{p_2},\xi_{p_3},x,y)$ over all $(\xi_{p_1},\xi_{p_2},\xi_{p_3})\in \mathcal{A}(\psi)$. By \textbf{R3}, we have that
$$
\mathbb{P}(F(\mathcal{A}(\psi),x,y))\geq \varepsilon^D \mathbb{P}(\mathcal{A}(\psi)).
$$

Moreover, we claim that for each $(\psi,x,y)$ there are at most $|B_{2D}(o)|^2$ triples $(\psi',z,w)$ such that $F(\mathcal{A}(\psi'),z,w)=F(\mathcal{A}(\psi),x,y)$. Indeed, $\xi_{p_1}$ and $\xi_{p_2}$ remain the same after applying $F$, which implies that from the event $F(\mathcal{A}(\psi),x,y)$, we can uniquely determine the sets $\mathcal{C}_{\infty}^{p_1,p_2}\cap B_N(o)$, $\mathcal{S}$ and $\overline{\mathcal{S}}^D$. Furthermore, since $\xi_{p_3}$ is only modified in $\overline{\mathcal{S}}^D\setminus\mathcal{S}$, we can uniquely determine $\tilde{\xi}_{p_3}$, and hence $\mathcal{C}_{o}^{p_3}\cap B_N(o)$. This implies that $\psi'=\psi$. It remains to determine the number of possibilities for $z,w$. Consider a pre-image $(\psi,z_0,w_0)$ of $F(\mathcal{A}(\psi),x,y)$, which means that on the event $F(\mathcal{A}(\psi),x,y)$ we have that $\xi'_{p_3}(v)\geq t$ for all $v\in \gamma_{z_0,w_0}\setminus \{z_0,w_0\}$. Note that for any $z,w$ such that $F(\mathcal{A}(\psi),z,w)=F(\mathcal{A}(\psi),x,y)$, $\gamma_{z,w}$ must intersect $\gamma_{z_0,w_0}$, since otherwise each vertex of $\gamma_{z_0,w_0}$ would be open in the original $\omega^{p_3}$, thus connecting $\mathcal{C}_{\infty}^{p_1,p_2}$ and $\mathcal{C}^{p_3}_o$. Hence, $z$ and $w$ must lie in $B_{2D}(z_0)$. Combining the above, we deduce that there are at most $|B_{2D}(o)|^2$ pre-images.

Let us write $\mathcal{B}$ to denote the event $\bigcup_{(\psi,x,y)} F(\mathcal{A}(\psi),x,y)$. For each $\psi$, let $\Sigma(\psi)$ denote a set of representative pairs of points $(z,w)$, i.e.\ for every distinct $(z,w), (z',w')\in \Sigma(\psi)$ we have $F(\mathcal{A}(\psi),z,w)\neq F(\mathcal{A}(\psi),z',w')$ and for each $(x,y)$ there exists $(z,w)\in \Sigma(\psi)$ such that $F(\mathcal{A}(\psi),z,w)= F(\mathcal{A}(\psi),x,y)$. We can now deduce that 
\begin{align*}
1&\geq \mathbb{P}(\mathcal{B})=\sum_{\psi}\sum_{(z,w)\in \Sigma(\psi)} \mathbb{P}(F(\mathcal{A}(\psi),z,w))\geq \frac{\varepsilon^D}{|B_{2D}(o)|^2} \sum_{(\psi,x,y)} \mathbb{P}(\mathcal{A}(\psi))\\&\geq  \frac{\varepsilon^D K}{|B_{2D}(o)|^2}\sum_{\psi} \mathbb{P}(\mathcal{A}(\psi))=\frac{\varepsilon^D K}{|B_{2D}(o)|^2}\mathbb{P}(\mathcal{A}).
\end{align*}
We now choose $K$ to be large enough so that
$$
\mathbb{P}(\mathcal{A})\leq \frac{|B_{2D}(o)|^2}{\varepsilon^D K}<c/2.
$$
Since $\mathcal{E}_{N,K}\cap \mathcal{E}\subset\mathcal{A}$, this contradicts \eqref{eq: contradiction}. This implies that almost surely, every $\omega^{p_3}$ infinite cluster contains $\mathcal{C}_{\infty}^{p_1,p_2}$, which in turn implies the desired uniqueness of infinite cluster.
\end{proof}

\begin{remark}
Beyond $\mathbb{Z}^d$, Theorem~\ref{thm:main result} applies to every vertex-transitive amenable graph. Indeed, our proof utilises the Burton-Keane argument, which applies to any vertex-transitive amenable graph, and the mass-transport principle. Graphs that satisfy the latter property are sometimes called \emph{unimodular}. It is known that amenability implies unimodularity \cite[Proposition 8.14]{LyonsPeres}.
\end{remark}

\begin{remark}
If we assume that each $\mathbb{P}_p$ is a product measure, then one can avoid the use of the multi-valued map principle in the proof of Theorem~\ref{thm:main result} and instead use independence to open the geodesics $\gamma_{x,y}$. 
\end{remark}

\bibliographystyle{abbrv}
\bibliography{references}

\end{document}